\documentclass[11pt,reqno]{amsart}

\usepackage{amscd,amssymb,amsmath,amsthm}
\usepackage[arrow,matrix]{xy}
\usepackage{graphicx}
\usepackage{epstopdf}
\usepackage{color}
\newtheorem{thm}{Theorem}
\newtheorem{defn}{Definition}

\newtheorem{rk}{Remark}

\numberwithin{equation}{section} 

\begin{document}

\title[Quadratic stochastic operator with two discontinuity points]{Dynamical system of a quadratic stochastic operator with two discontinuity points}

\author{ Sh.B. Abdurakhimova, U.A. Rozikov}

 \address{U.\ A.\ Rozikov \begin{itemize}
 \item[] V.I.Romanovskiy Institute of Mathematics of Uzbek Academy of Sciences;
\item[] AKFA University, 1st Deadlock 10, Kukcha Darvoza, 100095, Tashkent, Uzbekistan;
\item[] Faculty of Mathematics, National University of Uzbekistan.
\end{itemize}} \email {rozikovu@yandex.ru}

\address{Sh.B. Abdurakhimova. Namangan State University, Namangan,  Uzbekistan.}
\email {shakhnoza.karimova95@mail.ru}

\begin{abstract} In this paper we consider a population consisting of two species, dynamics of which is defined by a quadratic stochastic operator with
variable coefficients, making it discontinuous operator at two points. This operator depends on three parameters.
It is shown that under suitable conditions on the parameters this operator may have fixed points, convergence of trajectories and there may exist periodic points.
\end{abstract}

 \maketitle

\section{Introduction}

The concept of quadratic stochastic operator (QSO) was first introduced by Bernstein. A lot of papers were devoted to study such operators (see
for example \cite{GMR}-\cite{RZ1} and references therein). The theory of QSO frequently arises in many models of physics, biology, economics and
so on.

Let us give some necessary definitions (see \cite{D}, \cite{Rb}). In order to define a discrete-time dynamical system consider a function
$f:X\rightarrow{X}$.
    For $x \in {X}$ denote by $f^{n}{(x)}$ the $n$-fold composition of $f$ with itself:
\begin{center}\text{$f^{n}(x)=\underbrace{f(f(...f(x)...))}_{n}$}
\end{center}
\begin{defn}For arbitrary given $x^{(0)}\in {X}$ and $f:X\rightarrow{X}$ the discrete-time dynamical system (also called
forward orbit or trajectory of $x^{(0)}$) is the sequence of points
\begin{equation}\label{1.1}
x^{(0)}, x^{(1)}=f(x^{(0)}), x^{(2)}=f^{2}(x^{(0)}), x^{(3)}=f^{3}(x^{(0)}),\dots
\end{equation}
\end{defn}
\begin{defn}A point $x\in X$ is called a fixed point for $f:X\rightarrow X$ if $f(x)=x$. The point $x$ is called a
periodic point of period $p$ if $f^{p}(x)=x$. The least positive $p$ for which $f^{p}(x)=x$ is called the prime period of $x$.
\end{defn}

    It is clear that the set of all iterates of a periodic point form a periodic sequence (orbit).

    There are three kinds of periodic points: attracting, repelling and indifferent. Let $x^{*}$ be a $p$-periodic point. If
    $|(f^{p}{(x^{*})})'|<1$, $x^{*}$-attracting; $|(f^{p}{(x^{*})})'|>1$, $x^{*}$-repelling; $|(f^{p}{(x^{*})})'|=1$,
    $x^{*}$-indifferent.

Consider a biological population, that is, a community of
organisms closed with respect to reproduction. Let $S^{m-1}$ be the simplex:
\begin{equation}\label{1.2}
S^{m-1}=\{x=(x_{1},...,x_{m})\in \mathbb R^{m}: x_{i}\geq0, \sum\limits_{i=1}^{m} x_{i}=1\}.
\end{equation}
Assume that every individual in this population belongs to one of the species $1,2,...,m$. Let
$x^{0}=(x_{1}^{0},...,x_{m}^{0})\in{S^{m-1}}$ be the probability distribution (where $x_{i}^{0}=P(i)$ is the probability of
$i=1,2,...,m$) of species in the initial generation, and $P_{ij,k}$ the probability that individuals in the $i$th and $j$th species
interbreed to produce an individual $k$, more precisely $P_{ij,k}$ is the conditional probability $P(k|i,j)$ that $i$th and $j$th species
interbred successfully, then they produce an individual $k$.

Assume the parents $ij$ are independent i.e., $P(i,j)=P(i)P(j)=x_{i}^{0}x_{j}^{0}$. Then the probability distribution $x_k'$ of the
species in the first generation can be found by the total probability
\begin{equation}\label{1.3}
x_k'=\sum_{i,j=1}^{m}P(k|i,j)P(i,j)=\sum_{i,j=1}^{m}P_{ij,k}x_{i}^{0}x_{j}^{0},                   k=1,2,...,m
\end{equation}
This means that the association  $x^{0}\in S^{m-1}\rightarrow x'\in S^{m-1}$ defines a map $V$ called the \emph{evolution
operator}.

The states of the population described by the following discrete-time dynamical system
\begin{equation}\label{1.4}
x^{(0)}, x^{(1)}=V(x^{(0)}), x^{(2)}=V^{2}(x^{(0)}), ... , x^{(n)}=V^{n}(x^{(0)}),...
\end{equation}
where $V^{n}(x)=\underbrace{V(V(...V(x)...))}_{n}$ denotes  the $n$ times iteration of $V$ to $x$.

\textbf{The main problem} for a given dynamical system is to describe the limit points of $\{x^{(n)}\}_{n=0}^{\infty}$ for arbitrary
given $x^{(0)}$.

In this work we investigate the dynamics of a quadratic stochastic operator $V$ with variable coefficients, that is mapping $S^{1}$ into itself
(see \cite{RU}, \cite{UK} and references therein for motivations of such investigations).

 \section{A QSO with two discontinuity points}

 Consider a population of 2 species, i.e. $m=2$. Denote the set of species by $E=\{1, 2\}$.

    For a variable coefficient $p(x)$ define an evolution operator
$V_{a,b,c}:z=(x,y)\in S^{1}\rightarrow z'=(x',y')\in S^{1}$ as the following:
\begin{equation}\label{1.5}
V\equiv V_{a,b,c}:\left\{
\begin{array}{ll}
x'=x^{2}+2p(x)xy,\\[2mm]
y'=2(1-p(x))xy+y^{2}
\end{array}
\right.
\end{equation}
where
\\$$p(x)=\left\{
\begin{array}{ll}
a,\ \ x\leq \frac{1}{3},\\[2mm]
b,\ \  \frac{1}{3}<x<\frac{2}{3},\\[2mm]
c,\ \ x\geq \frac{2}{3}
\end{array}
\right.$$
with $a,b,c\in [0,1]$. We are interested to study the dynamical system generated by the evolution operator $V_{a,b,c}$.

Using $x+y=1$ the operator (\ref{1.5}) can be reduced to the mapping $f_{a,b,c}:[0,1]\rightarrow [0,1]$ defined by
\begin{equation}\label{2.1}
f_{a,b,c}(x)=\left\{
\begin{array}{ll}
(1-2a)x^{2}+2ax,\ \ x\leq \frac{1}{3},\\[2mm]
(1-2b)x^{2}+2bx,\ \  \frac{1}{3}<x<\frac{2}{3},\\[2mm]
(1-2c)x^{2}+2cx,\ \ x\geq \frac{2}{3}
\end{array}
\right.
\end{equation}
here $a,b,c\in [0,1]$.
\begin{figure}
\includegraphics[width=9cm]{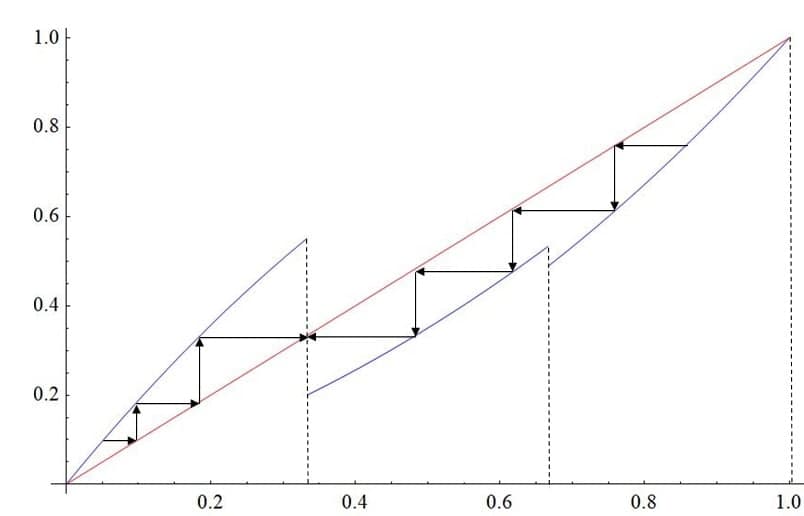}
\caption{The graph of the function (\ref{2.1}) for $a>1/2$ and $b<1/2$, $c<1/2$. Some trajectories are shown.}\label{f1}
\end{figure}

 Note that in the case $a=b=c={1\over 2}$ the function is  trivial, because $ f_{a,b,c}(x)=x$, $x\in [0,1].$ Therefore, below
 we do not consider this case.

 In case  $a=b=c\ne {1\over 2}$, we have $f(x):=f_{a,a,a}(x)=(1-2a)x^{2}+2ax$, $x\in [0,1]$.
 The fixed points of this function are $0$ and $1$.
 
It is easy to see that

1) The point $0$ is attracting for $f$, if $a \in [0,\frac{1}{2})$; repelling if $a \in (\frac{1}{2},1]$.

2) The point $1$ is repelling for $f$, if $a \in [0,\frac{1}{2})$; attracting if $a\in (\frac{1}{2},1]$.

Moreover, using monotonicity of $f$ one can show that for any initial point $x^{(0)}\in (0,1)$ the following limits hold
$$\lim_{n\to \infty}f^n(x^{(0)})=\left\{\begin{array}{ll}
0, \ \ \mbox{if} \ \ a \in [0,\frac{1}{2})\\[2mm]
1, \ \ \mbox{if} \ \ a \in (\frac{1}{2},1].
\end{array}
\right.
$$

 Thus value $1/2$ is a critical point where the character of the dynamical system changes.
 Therefore, we consider the following possible cases (see Fig. \ref{f1}):\\

1. $a\neq\frac{1}{2}$, $b=\frac{1}{2}$, $c=\frac{1}{2}$

2. $a=\frac{1}{2}$, $b\neq\frac{1}{2}$, $c=\frac{1}{2}$

3. $a=\frac{1}{2}$, $b=\frac{1}{2}$, $c\neq\frac{1}{2}$

4. $a\neq\frac{1}{2}$, $b\neq\frac{1}{2}$, $c=\frac{1}{2}$

5. $a\neq\frac{1}{2}$, $b=\frac{1}{2}$, $c\neq\frac{1}{2}$

6. $a=\frac{1}{2}$, $b\neq\frac{1}{2}$, $c\neq\frac{1}{2}$

7. $a\neq\frac{1}{2}$, $b\neq\frac{1}{2}$, $c\neq\frac{1}{2}$\\

\textbf{The case:} $a\neq\frac{1}{2}$, $b=\frac{1}{2}$, $c=\frac{1}{2}$.

In this case the function is
\begin{equation}\label{2.2}
f_{a}(x)=\left\{
\begin{array}{ll}
(1-2a)x^{2}+2ax,\ \ x\in [0,\frac{1}{3}],\\[2mm]
x,\ \ x\in (\frac{1}{3},1]
\end{array}
\right.
\end{equation}

\begin{thm}\label{t21} For the dynamical system generated by (\ref{2.2}) the following assertions hold
\begin{itemize}
\item[1.] The set of fixed points  is ${\rm Fix}(f_a)=\{0\}\cup (\frac{1}{3},1]$;

\item[2.] If $0\leq a <\frac{1}{2}$ then $\lim_{n\rightarrow\infty}x^{(n)}
=\lim_{n\rightarrow\infty}f_a^n(x^{(0)})=0$, for any $x^{(0)}\in [0,\frac{1}{3}]$;

\item[3.] If $\frac{1}{2}<a\leq1$ then for any $ x^{(0)}\in (0,\frac{1}{3}]$,
there exist $n\in \mathbb N$ and $p \in (\frac{1}{3};\frac{1}{9}(4a+1)]$
such that $f^{n}(x^{(0)})=p, f^{n+1}(x^{(0)})=f(p)=p$.
\end{itemize}
\end{thm}
\begin{proof}

1. Follows from a simple analysis of the equation of $f_a(x)=x$.

2. The function $f_a(x)$ is monotone increasing and $f_a(x)<x$,  $0\leq f_a(x) \leq \frac{1}{9}(1+4a)$ for any point $x \in [0,\frac{1}{3}]$. Therefore the orbit $f^n(x)$ is decreasing and bounded from below by the fixed point $x=0$. Hence the limit of this sequence is $0$.

3. In this case the function is increasing, $f_a(x)>x$ and $0\leq f_a(x) \leq \frac{1}{9}(1+4a)$ for any point $x \in [0,\frac{1}{3}]$. So, by the $1^{st}$ part of the theorem any point $p\in (\frac{1}{3},1]$ is a fixed point for function $f_a(x)$. Consequently, there is $n$ such that $f^{n}(x^{(0)})=p, f^{n+1}(x^{(0)})=f(p)=p$.
\end{proof}

\textbf{The case:} $a=\frac{1}{2}$, $b\neq\frac{1}{2}$, $c=\frac{1}{2}$.
In this case the function is
\begin{equation}\label{2.3}
f_{b}(x)=\left\{
\begin{array}{ll}
x,\ \ x\in [0,\frac{1}{3}]\cup [\frac{2}{3},1],\\[2mm]
(1-2b)x^{2}+2bx,\ \ x\in (\frac{1}{3},\frac{2}{3})
\end{array}
\right.
\end{equation}

\begin{thm}\label{t22} The dynamical system generated by the function (\ref{2.3}) has the following properties
\begin{itemize}
\item[1.] The set of fixed points is ${\rm Fix}(f_b)=[0,\frac{1}{3}]\cup[\frac{2}{3},1]$;

\item[2.] If $0\leq b <\frac{1}{2}$ then for any $x^{(0)}\in (\frac{1}{3},\frac{2}{3})$,
there exist $n \in \mathbb N$ and $p \in [\frac{1}{9}(1+4b);\frac{1}{3}]$ such that
$f^{n}(x^{(0)})=p, f^{n+1}(x^{(0)})=f(p)=p$;

\item[3.] If $\frac{1}{2}<b\leq1$ then for any $x^{(0)}\in (\frac{1}{3},\frac{2}{3})$,
there exist $n \in \mathbb N$ and $p \in [\frac{2}{3};\frac{4}{9}(1+b)]$
such that $f^{n}(x^{(0)})=p, f^{n+1}(x^{(0)})=f(p)=p$.
\end{itemize}
\end{thm}
\begin{proof} Is similar to the proof of Theorem \ref{t21}. \end{proof}

\textbf{The case:} $a=\frac{1}{2}$, $b=\frac{1}{2}$, $c\neq\frac{1}{2}$.
In this case the function is
\begin{equation}\label{2.4}
f_{c}(x)=\left\{
\begin{array}{ll}
x,\ \ x\in [0,\frac{2}{3}),\\[2mm]
(1-2c)x^{2}+2cx,\ \ x\in [\frac{2}{3},1]
\end{array}
\right.
\end{equation}

\begin{thm}\label{t23} For the dynamical system generated by function (\ref{2.4}) the following hold
\begin{itemize}
\item[1.] The set of fixed points is $Fix(f_c)=[0,\frac{2}{3})\cup\{1\}$;

\item[2.] If $0\leq c <\frac{1}{2}$ then for any $x^{(0)}\in [\frac{2}{3},1]$, there exist $n \in \mathbb N$ and $p \in [\frac{4}{9}(1+c);\frac{2}{3})$ such that $f^{n}(x^{(0)})=p, f^{n+1}(x^{(0)})=f(p)=p$;

\item[3.] If $\frac{1}{2}<c\leq1$ then $\lim_{n\to \infty}x^{(n)}=\lim_{n\rightarrow\infty}f_c^n(x^{(0)})=1$, $\forall x^{(0)}\in [\frac{2}{3},1]$.
\end{itemize}
\end{thm}
\begin{proof} Similar to the above cases. \end{proof}

\textbf{The case:} $a\neq\frac{1}{2}$, $b\neq\frac{1}{2}$, $c=\frac{1}{2}$.

In this case the function is
\begin{equation}\label{2.5}
f_{a,b}(x)=\left\{
\begin{array}{ll}
(1-2a)x^{2}+2ax,\ \ x\in [0,\frac{1}{3}],\\[2mm]
(1-2b)x^{2}+2bx,\ \ x\in (\frac{1}{3},\frac{2}{3}),\\[2mm]
x,\ \ x\in [\frac{2}{3},1].
\end{array}
\right.
\end{equation}

\begin{thm}\label{t24} For the dynamical system generated by function (\ref{2.5}) the following assertions hold
\begin{itemize}
\item[1.] The set of fixed points is ${\rm Fix}(f_{a,b})=[\frac{2}{3},1]\cup \{0\}$;

\item[2.] If $0\leq a <\frac{1}{2}$, $0\leq b <\frac{1}{2}$ then $\lim_{n\rightarrow\infty}x^{(n)}=0$, $\forall x^{(0)}\in [0,\frac{2}{3})$;

\item[3.] If $0\leq a <\frac{1}{2}$, $\frac{1}{2}<b\leq1$ then $\lim_{n\rightarrow\infty}x^{(n)}=0$, $\forall x^{(0)}\in [0,\frac{1}{3}]$;

\item[4.] If $0\leq a <\frac{1}{2}$, $\frac{1}{2}<b\leq1$ then for any $x^{(0)}\in (\frac{1}{3},\frac{2}{3})$, there exist $n \in \mathbb N$ and  $p \in [\frac{2}{3},\frac{4}{9}(1+b))$ such that $f^{n}(x^{(0)})=p, f^{n+1}(x^{(0)})=f(p)=p$;

\item[5.] If $\frac{1}{2}<a\leq1$, $\frac{1}{2}<b\leq1$ then for any $ x^{(0)}\in (0,\frac{2}{3})$, there exist $n \in \mathbb N$ and $p \in [\frac{2}{3},\frac{4}{9}(1+b))$ such that $f^{n}(x^{(0)})=p, f^{n+1}(x^{(0)})=f(p)=p$;

\item[6.] If $\frac{1}{2}<a\leq1$, $0\leq b <\frac{1}{2}$ then for sets
$$A_{1}=\left[0,\left.\frac{1}{9}(1+4b)\right.\right)\bigcup\left(\left.\frac{1}{9}(1+4a),\frac{2}{3}\right.\right), \ \ A_{2}=\left[\frac{1}{9}(1+4b),\frac{1}{9}(1+4a)\right]$$ the following hold:
\begin{itemize}
\item[1)] $\forall x_0\in A_{1}$  there exists $n_0(x_0)\in \mathbb N$, such that $f^n(x_0)\in A_{2}$ for any $n>n_0$;

\item[2)] $f(A_{2})\subset A_{2}$.
\end{itemize}
\end{itemize}
\end{thm}
\begin{proof}

Parts 1-5 are similar to corresponding parts of the previous theorems.

6. 1) For $x \in A_1$ we have $x \in [0,\frac{1}{9}(1+4b))$ or $x \in (\frac{1}{9}(1+4a),\frac{2}{3})$. Let's suppose $x \in [0,\frac{1}{9}(1+4b))$. We solve the inequality $f_{a,b}(x)>x$, in this case the function is $f_{a,b}(x)=(1-2a)x^{2}+2ax \Rightarrow (1-2a)x(x-1)>0$ and its solution is any $x \in (0,1)$.
Moreover we have $$ f_{a,b}(x)\in \left(\frac{1}{9}(1+4a),\frac{2}{3}\right) \Rightarrow f_{a,b}(x_1)>x_1, \ \  x_1=f_{a,b}(x)$$
 $$\Rightarrow (1-2a)^2x^2(x-1)((1-2a)x+2a)>0 \Rightarrow x\in \left(1,\frac{2a}{2a-1}\right)$$ this solution does not belong in the interval $[0,1]$.

Let's suppose the trajectory of  $x_0 \in [0,\frac{1}{9}(1+4b))$, i.e., $x_n=f^n(x_0)$ does not go inside of $A_2$, then this trajectory has its own limit because it is an increasing and bounded by $\frac{1}{9}(1+4b)$ (i.e. the left endpoint of $A_2$). But $\frac{1}{9}(1+4b)$ isn't fixed point and there is no any fixed point in the left side of $A_2$. Thus, this trajectory goes inside of $A_2$.

For $x\in (\frac{1}{9}(1+4a),\frac{2}{3})$ we have
$$f_{a,b}(x)<x, \ \ f_{a,b}(x)=(1-2b)x^{2}+2bx \Rightarrow (1-2b)x(1-x)<0 \Rightarrow x\in (0,1).$$
Moreover, if
$f_{a,b}(x)\in (\frac{1}{9}(1+4a),\frac{2}{3})$ then, denoting $x_1=f_{a,b}(x)$, the inequality
$f_{a,b}(x_1)>x_1$ can be reduced to
$(1-2b)^2x^2(x-1)((1-2b)x+2b)>0$ which has solutions $x\in (-\infty,\frac{2b}{2b-1})\bigcup(1,\infty)$. For any point in the right of $A_2$ one can do similar arguments.

2) We prove that $f(x)\in A_2$ for all $x\in A_2$. For $x\in A_2$ we have
$$x\in \left[\frac{1}{9}(1+4b),\frac{1}{3}\right]\bigcup\left(\frac{1}{3},\right.\left.\frac{1}{9}(1+4a)\right].$$
 Let's suppose, $x\in [\frac{1}{9}(1+4b),\frac{1}{3}]$. It is easy to check that, $f(\frac{1}{9}(1+4a))$ and $f(\frac{1}{3})$ are elements of $A_2$. Then $f(x)\in A_2$ because $f(x)$ is monotone increasing. For the second case, i.e. $x\in (\frac{1}{3},\frac{1}{9}(1+4a)]$ we may check similarly. Therefore, $f(A_2)\subset A_2$. Theorem is proved. \end{proof}

Let us study 2-periodic points of the function (\ref{2.5}). It is clear that a 2-periodic orbit $x_1$, $x_2$ exists if they
satisfy the following system of equations:
\begin{equation}\label{pe}
(1-2a)x_1^2+2ax_1=x_2, \ \ (1-2b)x_2^2+2bx_2=x_1.
\end{equation}
The following theorem gives existence conditions of 2-periodic points.
\begin{thm}\label{pt} For each pair $x_1, x_2$ satisfying
\begin{equation}\label{pc}
x_1\in \left(1-{\sqrt{6}\over 3}, {1\over 3}\right), \ \ x_2\in \left({1\over 3}, x_1(2-x_1)\right)
\end{equation}
there exists unique pair $a, b$ such that $x_1, x_2$ is 2-periodic orbit for the function $f_{a,b}$ (given by (\ref{2.5})).
\end{thm}
\begin{proof}
Substituting $x_1$ from the second equation of (\ref{pe}) to the first equation we get a polynomial
equation of order 4. Which may have up to four solutions, two of them are $x_2=0$ and $x_2=1$ independently on the values of parameters. These solutions are fixed points. To find 2-periodic (except
fixed) points we divide the polynomial
to $x_2$ and $x_2-1$ and get a quadratic equation, discriminant of which has the following form
$$D=(2a-1)(2b-1)[(2a-1)(2b-1)-4].$$
Denote $t=(2a-1)(2b-1)$. Note that $t<1$ therefore $D=t(t-4)\geq 0$ if and only if $t\leq 0$.
Since we have condition that $a\ne 1/2$ and $b\ne 1/2$, the condition $t< 0$ is satisfied if and only if $a<1/2$, $b>1/2$ or $a>1/2$, $b<1/2$.
It is easy to see that only in the case $a>1/2$, $b<1/2$ there may exist periodic points.

Solving now (\ref{pe}) with respect to $a$ and $b$ we get
\begin{equation}\label{ab}
a={x_2-x_1^2\over 2x_1(1-x_1)}, \ \ b={x_1-x_2^2\over 2x_2(1-x_2)}.
\end{equation}
From these equalities note that for given $x_1$ and $x_2$ the values $a$ and $b$ are uniquely defined.

For (\ref{ab}) one can see that the condition $1/2<a<1$, $0<b<1/2$ is equivalent to the condition (\ref{pc}). Moreover, since (\ref{ab})
is equivalent to (\ref{pe}) we conclude that $x_1$, $x_2$ is 2-periodic orbit for function $f_{a,b}$.

Inversely, for given $a, b$ satisfying $1/2<a<1$, $0<b<1/2$ the exact formula for $x_1$ and $x_2$ is given by
$$x_1=f_{a,b}(x_2), \ \ x_2={t-\sqrt{D}\over 2t}.$$
 \end{proof}
 A numerical analysis shows that if, for example, $x_1=0.2$ and $x_2=0.34$ then  $a=0.9375$, $b=0.188057041$.

\begin{rk} Theorem \ref{pt} gives 2-periodic points on a neighborhood of the discontinuity point $1/3$. In the case $b>1/2$ and $c<1/2$ a similar result can be proved for neighborhood of $2/3$. 

\end{rk}

\textbf{The case:} $a\neq\frac{1}{2}$, $b=\frac{1}{2}$, $c\neq\frac{1}{2}$.

In this case the function is
\begin{equation}\label{2.6}
f_{a,c}(x)=\left\{
\begin{array}{ll}
(1-2a)x^{2}+2ax,\ \ x\in [0,\frac{1}{3}],\\[2mm]
x,\ \ x\in (\frac{1}{3},\frac{2}{3}),\\[2mm]
(1-2c)x^{2}+2cx,\ \ x\in [\frac{2}{3},1]
\end{array}
\right.
\end{equation}

\begin{thm}\label{t25} The dynamical system generated by function (\ref{2.6}) has the following properties
\begin{itemize}
\item[1.] The set of fixed points  is ${\rm Fix}(f_{a,c})=(\frac{1}{3},\frac{2}{3})\cup \{0;1\}$;

\item[2.] If $0\leq a <\frac{1}{2}$, $0\leq c <\frac{1}{2}$ then for any $ x^{(0)}\in [\frac{2}{3},1)$, there exist $n \in \mathbb N$ and $p \in [\frac{4}{9}(1+c),\frac{2}{3})$ such that $f^{n}(x^{(0)})=p, f^{n+1}(x^{(0)})=f(p)=p$;

\item[3.] If $0\leq a <\frac{1}{2}$,  then (independently on $c$) $\lim_{n\rightarrow\infty}x^{(n)}=0$, $\forall x^{(0)}\in [0,\frac{1}{3}]$;

\item[4.] If  $\frac{1}{2}<c\leq1$ then (independently on $a$) $\lim_{n\rightarrow\infty}x^{(n)}=1$, $\forall x^{(0)}\in [\frac{2}{3},1]$;

\item[5.] If $\frac{1}{2}<a\leq1$, $0\leq c <\frac{1}{2}$ then for any $ x^{(0)}\in (0,\frac{1}{3}]$ (resp. $x^{(0)}\in [\frac{2}{3},1)$) there exist $n \in \mathbb N$ and $p \in (\frac{1}{3},\frac{1}{9}(4a+1))$ (resp. $p \in [\frac{4}{9}(1+c),\frac{2}{3})$) such that $f^{n}(x^{(0)})=p, f^{n+1}(x^{(0)})=f(p)=p$;

\item[6.] If $\frac{1}{2}<a\leq1$, $\frac{1}{2}<c\leq1$ then for any $x^{(0)}\in (0,\frac{1}{3}]$, there exist $n \in \mathbb N$ and $p \in (\frac{1}{3},\frac{1}{9}(4a+1)]$ such that  $f^{n}(x^{(0)})=p, f^{n+1}(x^{(0)})=f(p)=p$;

\end{itemize}
\end{thm}
\begin{proof} The proof of this theorem is similar to the proof of Theorem \ref{t24}.\end{proof}

\textbf{The case:} $a=\frac{1}{2}$, $b\neq\frac{1}{2}$, $c\neq\frac{1}{2}$.

In this case the function is
\begin{equation}\label{2.7}
f_{b,c}(x)=\left\{
\begin{array}{ll}
x,\ \ x\in [0,\frac{1}{3}],\\[2mm]
(1-2b)x^{2}+2bx,\ \ x\in (\frac{1}{3},\frac{2}{3}),\\[2mm]
(1-2c)x^{2}+2cx,\ \ x\in [\frac{2}{3},1]
\end{array}
\right.
\end{equation}

\begin{thm}\label{t26} For the dynamical system generated by (\ref{2.7}) we have
\begin{itemize}
\item[1.] The set of fixed points is ${\rm Fix}(f_{b,c})=[0,\frac{1}{3}]\cup \{1\}$;

\item[2.] If $0\leq b <\frac{1}{2}$, $0\leq c <\frac{1}{2}$ then for any $x^{(0)}\in (\frac{1}{3},1)$, there exist $n \in \mathbb N$ and $p \in [\frac{1}{9}(4b+1),\frac{1}{3}]$ such that $f^{n}(x^{(0)})=p, f^{n+1}(x^{(0)})=f(p)=p$;

\item[3.] If $0\leq b <\frac{1}{2}$, $\frac{1}{2}<c\leq1$ then for any $x^{(0)}\in (\frac{1}{3},\frac{2}{3})$, there exist $n \in \mathbb N$ and $p \in [\frac{1}{9}(4b+1),\frac{1}{3}]$ such that $f^{n}(x^{(0)})=p, f^{n+1}(x^{(0)})=f(p)=p$;

\item[4.] If $0\leq b <\frac{1}{2}$, $\frac{1}{2}<c\leq1$ then $\lim_{n\rightarrow\infty}x^{(n)}=1$, $\forall x^{(0)}\in [\frac{2}{3},1]$;

\item[5.] If $\frac{1}{2}<b\leq1$, $\frac{1}{2}<c\leq1$ then $\lim_{n\rightarrow\infty}x^{(n)}=1$, $\forall x^{(0)}\in (\frac{1}{3},1]$;

\item[6.] If  $\frac{1}{2}<b\leq1$, $0\leq c <\frac{1}{2}$, then for sets
$$A_{1}=\left(\frac{1}{3}\right.,\left.\frac{4}{9}(1+c)\right)\bigcup\left(\frac{4}{9}(1+b)\right., 1\bigg],
\ \ A_{2}=\left[\frac{4}{9}(1+c), \frac{4}{9}(1+b)\right]$$ the followings hold:
\begin{itemize}
\item[1)] $\forall x_0\in A_1$ there exists $n_0(x_0)\in \mathbb N$, such that $f^n(x_0)\in A_{2}$  for any $n>n_0$;

\item[2)] $f(A_{2})\subset A_{2}$;
\end{itemize}
\end{itemize}
\end{thm}
\begin{proof} Similar to the proof of Theorem \ref{t24}.\end{proof}

\textbf{The case:} $a\neq\frac{1}{2}$, $b\neq\frac{1}{2}$, $c\neq\frac{1}{2}$.

In this case the function is
\begin{equation}\label{2.8}
f_{a,b,c}(x)=\left\{
\begin{array}{ll}
(1-2a)x^{2}+2ax,\ \ x\in [0,\frac{1}{3}],\\[2mm]
(1-2b)x^{2}+2bx,\ \ x\in (\frac{1}{3},\frac{2}{3}),\\[2mm]
(1-2c)x^{2}+2cx,\ \ x\in [\frac{2}{3},1]
\end{array}
\right.
\end{equation}

\begin{thm}\label{t27} The dynamical system generated by function (\ref{2.8}) has the following properties
\begin{itemize}
\item[1.] The set of fixed points is ${\rm Fix}(f_{a,b,c})=\{0;1\}$;

\item[2.] If $0\leq a <\frac{1}{2}$, $0\leq b <\frac{1}{2}$, $0\leq c <\frac{1}{2}$ then $\lim_{n\rightarrow\infty}x^{(n)}=0$, $\forall x^{(0)}\in [0,1]$;

\item[3.] If $0\leq a <\frac{1}{2}$, $0\leq b <\frac{1}{2}$, $\frac{1}{2}< c \leq1$ then
$$\lim_{n\rightarrow\infty}x^{(n)}=\left\{\begin{array}{ll}
0,  \ \ \mbox{if} \ \  x^{(0)}\in [0,\frac{2}{3})\\[2mm]
1, \ \ \mbox{if} \ \  x^{(0)}\in [\frac{2}{3},1];
\end{array}\right.$$
\item[4.] If $0\leq a <\frac{1}{2}$, $\frac{1}{2}<b\leq1$, $0\leq c <\frac{1}{2}$ then $\lim_{n\rightarrow\infty}x^{(n)}=0$, $\forall x^{(0)}\in [0,\frac{1}{3}]$;

\item[5.] If $0\leq a <\frac{1}{2}$, $\frac{1}{2}<b\leq1$, $0\leq c <\frac{1}{2}$ then for sets
$$A_{1}=\left[0, \frac{1}{3}\right], \ \ A_{2}=\left(\frac{1}{3}, \frac{4}{9}(1+c)\right)\bigcup\left(\frac{4}{9}(1+b), 1\right), \ \ A_{3}=\left[\frac{4}{9}(1+c),\frac{4}{9}(1+b)\right]$$  the followings hold:
\begin{itemize}
\item[1)]  $\lim_{n\rightarrow\infty}x^{(n)}=0$  for any $x^{(0)}\in A_{1}$;

\item[2)] $\forall x_0\in A_2$ there exists $n_0(x_0)\in \mathbb N$, such that $f^n(x_0)\in A_{3}$ for any $n>n_0$;

\item[3)] $f(A_{3})\subset A_{3}$;
\end{itemize}

\item[6.] If $0\leq a <\frac{1}{2}$, $\frac{1}{2}<b\leq1$, $\frac{1}{2}<c\leq1$ then
$$\lim_{n\rightarrow\infty}x^{(n)}=\left\{\begin{array}{ll}
0,  \ \ \mbox{if} \ \  x^{(0)}\in [0,\frac{1}{3})\\[2mm]
1, \ \ \mbox{if} \ \  x^{(0)}\in [\frac{1}{3},1];
\end{array}\right.$$

\item[7.] If $\frac{1}{2}<a\leq1$, $0\leq b <\frac{1}{2}$, $0\leq c <\frac{1}{2}$, then for sets
$$B_{1}=\bigg[0, \frac{1}{9}(1+4b)\bigg)\bigcup\bigg(\frac{1}{9}(1+4a), 1\bigg], \ \
B_{2}=\left[\frac{1}{9}(1+4b), \frac{1}{9}(1+4a)\right]$$  the followings hold:
\begin{itemize}
\item[1)] $\forall x_0\in B_1$ there exists $n_0(x_0)\in \mathbb N$, such that $f^n(x_0)\in B_{2}$ for any $n>n_0$;

\item[2)] $f(B_{2})\subset B_{2}$;
\end{itemize}

\item[8.] If $\frac{1}{2}<a\leq1$, $0\leq b <\frac{1}{2}$, $\frac{1}{2}<c\leq1$, then  for sets  $$C_{1}=\bigg[0, \frac{1}{9}(1+4b)\bigg)\bigcup\left(\frac{1}{9}(1+4a), \frac{2}{3}\right), \ \
    C_{2}=\left[\frac{1}{9}(1+4b), \frac{1}{9}(1+4a)\right], \ \
    C_{3}=\left[\frac{2}{3}, 1\right]$$  the followings hold:
\begin{itemize}
\item[1)] $\forall x_0\in C_1$ there exists $n_0(x_0)\in \mathbb N$, such that $f^n(x_0)\in C_{2}$ for any $n>n_0$;

\item[2)] $f(C_{2})\subset C_{2}$;

\item[3)] For $\forall x^{(0)}\in C_{3}$  $\lim_{n\rightarrow\infty}x^{(n)}=1$;
\end{itemize}
\item[9.] If $\frac{1}{2}<a\leq1$, $\frac{1}{2}<b\leq1$, $\frac{1}{2}<c\leq1$, then $\lim_{n\rightarrow\infty}x^{(n)}=1$, $\forall x^{(0)}\in (0,1]$.\\[3mm]
\end{itemize}
\end{thm}
\begin{proof} Similar to the above cases. \end{proof}

\begin{rk} {\rm Biological interpretations}: Let a biological system consists
$m$ species then each element of an $m$-dimensional simplex can be considered as a
state of the biological system, which is a probability distribution on the set of species consisting $m$ elements.
Therefore, the $i$-th coordinate $x_i^{(n)}$ of the vector $x^{(n)}$ is the probability to see $i$-th specie at the time $n$.
If the limit $\lim_{n\to \infty}x_i^{(n)}=\hat x_i$ exists then it means that the $i$-th specie has asymptotical probability $\hat x_i$.
Consequently, $\hat x_i=0$ biologically means that $i$-th specie will disappear (die) in the population. Moreover, $\hat x_i>0$
means that the $i$-th specie will survive. This general remark can be used to see biological meaning of each results given in above-mentioned
theorems.
\end{rk}

\end{document}